\documentclass[11pt]{article} 
 
\usepackage{epsfig}
\usepackage{amssymb}
\usepackage[]{times} 
\newcommand{\be}{\begin{equation}}
\newcommand{\ee}{\end{equation}} 
\newcommand{\bea}{\begin{eqnarray}}
\newcommand{\eea}{\end{eqnarray}}

\newtheorem{lemma}{Lemma} 

\newenvironment{proof}[1][Proof]{\begin{trivlist}
\item[\hskip\labelsep {\bfseries #1}]}{\end{trivlist}}
\def\1#1{^{(#1)}}
\def\la{\langle} 
\def\ra{\rangle}
\begin{document} 
\title{Semigroup's series for negative degrees of the gaps values in numerical 
semigroups generated by two integers\\
and identities for the Hurwitz zeta functions}
\author{Leonid G. Fel${}^{\dag}$ and Takao Komatsu${}^{\ddag}$\\ \\
${}^{\dag}$Department of Civil Engineering, Technion, Haifa 32000, Israel\\
{\em lfel@technion.ac.il}\\
${}^{\ddag}$School of Mathematics and Statistics, Wuhan University, Wuhan 
430072, China\\
{\em komatsu@whu.edu.cn}}
\date{} 
\maketitle
\def\be{\begin{equation}}
\def\ee{\end{equation}} 
\def\bea{\begin{eqnarray}}
\def\eea{\end{eqnarray}}
\def\p{\prime} 
\begin{abstract} 
We derive an explicit expression for an inverse power series over the gaps 
values of numerical semigroups generated by two integers. It implies a 
set of new identities for the Hurwitz zeta function.\\ 
{\bf Keywords:} numerical semigroups, gaps and non-gaps, the Hurwitz zeta 
function \\ 
{\bf 2010 Mathematics Subject Classification:} Primary -- 20M14, Secondary -- 
11P81. 
\end{abstract} 
\section{Introduction}\label{w1}
A sum of integer powers of the gaps values in numerical semigroups $S_m=\la d_1,
\ldots,d_m\ra$, where $\gcd(d_1,\ldots,d_m)=1$, is referred often as semigroup's
series 
\bea
g_n(S_m)=\sum_{s\in{\mathbb N}\setminus S_m}s^n,\qquad n\in{\mathbb Z},
\label{a1}
\eea
and $g_0(S_m)$ is known as a genus of $S_m$. For $n\ge 0$ an explicit expression
 of $g_n(S_2)$ and implicit expression of $g_n(S_3)$ were given in \cite{ro90} 
and \cite{fr07}, respectively. In this paper we derive a formula for semigroup 
series $g_{-n}(S_2)=\sum_{s\in{\mathbb N}\setminus S_m}s^{-n}$ ($n\ge 1$). 

Consider a numerical semigroup $S_2=\la d_1,d_2\ra$, generated by two integers 
$d_1,d_2\ge 2$, $\gcd(d_1,d_2)=1$, with the Hilbert series $H(z;S_2)$ and the 
gaps generating function $\Phi(z;S_2)$ given as follows.  
\bea
H(z;S_2)=\sum_{s\in S_2}z^s,\qquad\Phi(z;S_2)=\sum_{s\in{\mathbb N}\setminus 
S_2}z^s,\qquad H(z;S_2)+\Phi(z;S_2)=\frac1{1-z},\quad z<1,\label{a2}
\eea
where $\min\{{\mathbb N}\setminus S_2\}=1$ and $\max\{{\mathbb N}\setminus S_2\}
=d_1d_2-d_1-d_2$ is called the Frobenius number and is denoted by $F_2$. The 
rational representation (Rep) of $H(z;S_2)$ is given by
\bea
H(z;S_2)=\frac{1-z^{d_1d_2}}{(1-z^{d_1})(1-z^{d_2})}.\label{a3}
\eea
Introduce a new generating function $\Psi_1(z;S_2)$ by 
\bea
\Psi_1(z;S_2)=\int_0^z\frac{\Phi(t;S_2)}{t}dt=\sum_{s\in{\mathbb N}\setminus S_
2}\frac{z^s}{s},\qquad \Psi_1(1;S_2)=\sum_{s\in{\mathbb N}\setminus S_2}\frac1
{s}=g_{-1}(S_2).\label{a4}
\eea
Plugging (\ref{a2}) into (\ref{a4}), we obtain  
\bea
\Psi_1(z;S_2)=\int_0^z\left(\frac1{1-t}-H(t;S_2)\right)\frac{dt}{t}.\label{a5}
\eea
Keeping in mind $(1-t^{d_i})^{-1}=\sum_{k_i=0}^{\infty}t^{k_id_i}$, substitute 
the last into (\ref{a3}) and obtain
\bea
H(t;S_2)=\sum_{k_1,k_2=0}^{\infty}t^{k_1d_1+k_2d_2}-\sum_{k_1,k_2=0}^{\infty}
t^{k_1d_1+k_2d_2+d_1d_2}.\label{a6}
\eea
Indeed, an expression (\ref{a6}) is an infinite series with degrees $s=k_1d_1+
k_2d_2$ running over all nodes in the sublattice ${\mathbb K}$ of the integer 
lattice ${\mathbb Z}_2$ ,
\bea
{\mathbb K}=\{0,0\}\cup{\mathbb K}_1\cup{\mathbb K}_2,\qquad\left\{\begin{array}
{l}{\mathbb K}_1=\{1\le k_1\le d_2-1,\;\;k_2=0\},\\
{\mathbb K}_2=\{0\le k_1\le d_2-1,\;\;1\le k_2\le\infty\}.\end{array}\right.
\label{a7}
\eea
In Figure \ref{fig1}, we present, as an example, a part of the integer lattice
${\mathbb K}$ for the numerical semigroup $\la 5,8\ra$.
\begin{figure}[h!]\begin{center}
\psfig{figure=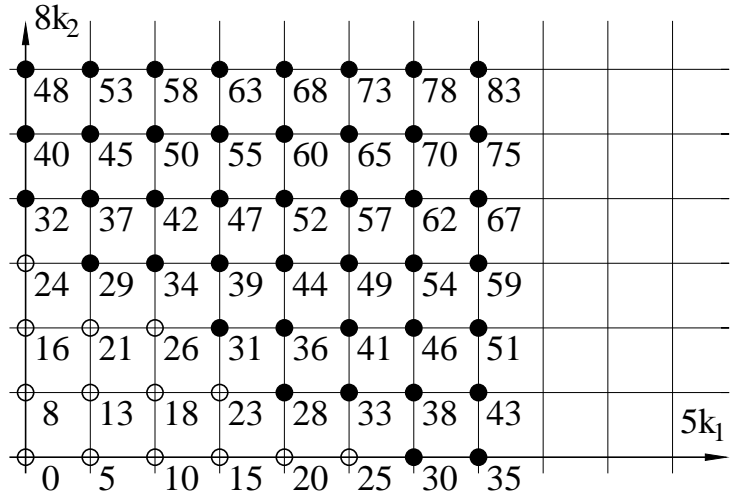,height=6.5cm}
\end{center}
\vspace{-.5cm}
\caption{A part of the integer lattice ${\mathbb K}\subset{\mathbb Z}_2$ for
semigroup $\la 5,8\ra$. The nodes mark the non-gaps of semigroup: the values,  
assigned to the black and white nodes, exceed and precede $F_2=27$, 
respectively.}\label{fig1}
\end{figure}
\begin{lemma}\label{le1}
There exists a bijection between the infinite set of nodes in the integer 
lattice ${\mathbb K}$  and an infinite set of non-gaps of the semigroup $\la 
d_1,d_2\ra$.
\end{lemma}
\begin{proof}
We have to prove two statements of existence and uniqueness:
\begin{enumerate}
\item Every $s\in\la d_1,d_2\ra$ has its Rep node in ${\mathbb K}$, 
\item All $s\in\la d_1,d_2\ra$ have their Rep nodes in ${\mathbb K}$ only once.
\end{enumerate}
1. Let $s\in\la d_1,d_2\ra$ be given. Then by definition of $\la d_1,d_2\ra$ an 
integer $s$ has a Rep,
\bea
s=k_1d_1+k_2d_2,\quad 0\le k_1,k_2\le\infty.\label{a8}
\eea
Choose $s$ such that $k_1=pd_2+q$, where $p=\left\lfloor k_1/d_2\right\rfloor$, 
i.e., $0\le q\le d_2-1$, and $\left\lfloor r\right\rfloor$ denotes an integer 
part of a real number $r$. Then Rep (\ref{a8}) reads,
\bea
s=qd_1+(k_2+pd_1)d_2,\nonumber
\eea
and $s$ has its Rep node in ${\mathbb K}$.

2. By way of contradiction, assume that there exist two nodes $\{k_1,k_2\}\in
{\mathbb K}$ and $\{l_1,l_2\}\in{\mathbb K}$ such that 
\bea
k_1d_1+k_2d_2=l_1d_1+l_2d_2,\qquad 0\le k_1,l_1\le d_2-1,\quad 0\le k_2,l_2\le
\infty,\quad k_1>l_1,\;\;k_2< l_2,\label{a9}
\eea
namely, there exists such $s\in\la d_1,d_2\ra$ which has two different Rep nodes
in ${\mathbb K}$. Rewrite (\ref{a9}) as 
\bea
(k_1-l_1)d_1=(l_2-k_2)d_2.\label{a10}
\eea
Since $\gcd(d_1,d_2)=1$, the equality (\ref{a10}) implies that 
\bea
k_1-l_1=bd_2,\quad b\ge 1\quad\longrightarrow\quad k_1=l_1+bd_2\quad
\longrightarrow\quad k_1\ge d_2,\label{a11}
\eea
that contradict the assumption $\{k_1,k_2\}\in{\mathbb K}$.$\;\;\;\;\;\;\Box$
\end{proof}
\section{A sum of the inverse gaps values $g_{-1}(S_2)$}\label{w2}
Present an integral in (\ref{a5}) as follows.  
\bea
\Psi_1(z;S_2)=\int_0^z\left(\sum_{k=0}^{\infty}t^{k-1}-\frac{H(t;S_2)}{t}\right)
dt,\qquad\frac{H(t;S_2)}{t}=\sum_{j=0}^2h_j(t;S_2),\label{a12}\\
h_0(t;S_2)=\frac1{t},\qquad h_1(t;S_2)=\sum_{k_1=1}^{d_2-1}t^{k_1d_1-1},\qquad 
h_2(t;S_2)=\sum_{k_1,k_2\in{\mathbb K}_2}t^{k_1d_1+k_2d_2-1}.\label{a13}
\eea
Perform an integration in (\ref{a12}),
\bea
\Psi_1(z;S_2)=\sum_{k=1}^{\infty}\frac{z^k}{k}-\frac1{d_1}\sum_{k_1=1}^{d_2-1}
\frac{z^{k_1d_1}}{k_1}-\sum_{k_1,k_2\in{\mathbb K}_2}\frac{z^{k_1d_1+k_2d_2}}
{k_1d_1+k_2d_2},\label{a14}
\eea
and obtain by (\ref{a4}) and (\ref{a7}), 
\bea
g_{-1}(S_2)=\sum_{k=1}^{\infty}\frac1{k}-\sum_{k_1,k_2\in{\mathbb K}_2}\frac1
{k_1d_1+k_2d_2}-\frac1{d_1}\sum_{k_1=1}^{d_2-1}\frac1{k_1}.\label{a15}
\eea
By Lemma \ref{le1}, after subtraction in (\ref{a15}) there are left a finite 
number of terms, since all terms, which exceed $F_2$ in the two first infinite 
series in (\ref{a15}), are cancelled. To emphasize that fact, we write as 
\bea
g_{-1}(S_2)=\sum_{k=1}^{c_2}\frac1{k}-\sum_{k_1,k_2\in{\mathbb K}_2}^{k_1d_1+
k_2d_2\le c_2}\frac1{k_1d_1+k_2d_2}-\frac1{d_1}\sum_{k_1=1}^{d_2-1}\frac1{k_1},
\label{a16}
\eea 
where $c_2=F_2+1$ is called a conductor of semigroup. 
\section{A sum of the negative degrees of gaps values $g_{-n}(S_2)$}\label{w3}
Generalize formula (\ref{a14}) and introduce a new generating function 
$\Psi_n(z;S_2)$ ($n\ge 2$) by
\bea
\Psi_n(z;S_2)=\int_0^z\frac{dt_1}{t_1}\int_0^{t_1}\frac{dt_2}{t_2}\ldots\int_0
^{t_{n-1}}\frac{dt_n}{t_n}\Phi(t_n;S_2)=\sum_{s\in{\mathbb N}\setminus S_2}
\frac{z^s}{s^n},\quad\Psi_n(1;S_2)=g_{-n}(S_2),\quad\label{a17}
\eea
which satisfies the following recursive relation.  
\bea
\Psi_{k+1}(t_{n-k-1};S_2)=\int_0^{t_{n-k-1}}\frac{dt_{n-k}}{t_{n-k}}\Psi_{k}
(t_{n-k};S_2),\;\;k\ge 0,\quad\Psi_0(t_n;S_2)=\Phi(t_{n-1};S_2),\;\;
t_0=z,\nonumber
\eea
namely,
\bea
\Psi_1(t_{n-1};S_2)=\int_0^{t_{n-1}}\frac{dt_{n}}{t_{n}}\Psi_0(t_{n};S_2),\quad
\Psi_2(t_{n-2};S_2)=\int_0^{t_{n-2}}\frac{dt_{n-1}}{t_{n-1}}\Psi_1(t_{n-1};S_2),
\quad \dots.\nonumber
\eea
Performing integration in (\ref{a17}), we obtain 
\bea
\Psi_n(z;S_2)=\sum_{k=1}^{\infty}\frac{z^k}{k^n}-\frac1{d_1^n}\sum_{k_1=1}^
{d_2-1}\frac{z^{k_1d_1}}{k_1^n}-\sum_{k_1,k_2\in{\mathbb K}_2}\frac{z^{k_1d_1
+k_2d_2}}{(k_1d_1+k_2d_2)^n}.\label{a18}
\eea
Thus, for $z=1$ we have 
\bea
g_{-n}(S_2)=\sum_{k=1}^{\infty}\frac1{k^n}-\sum_{k_1=0}^{d_2-1}\sum_{k_2=1}^{
\infty}\frac1{(k_1d_1+k_2d_2)^n}-\frac1{d_1^n}\sum_{k_1=1}^{d_2-1}\frac1{k_1^n},
\quad n\ge 2.\label{a19}
\eea
Define a ratio $\delta=d_1/d_2$ and represent the last expression as follows. 
\bea
g_{-n}(S_2)=\sum_{k=1}^{\infty}\frac1{k^n}-\frac1{d_2^n}\sum_{k_2=1}^{\infty}
\frac1{k_2^n}-\frac1{d_2^n}\sum_{k_1=1}^{d_2-1}\sum_{k_2=1}^{\infty}\frac1{
(k_1\delta+k_2)^n}-\frac1{d_1^n}\sum_{k_1=1}^{d_2-1}\frac1{k_1^n}.\nonumber
\eea
Making use of the Hurwitz $\zeta(n,q)=\sum_{k=0}^{\infty}(k+q)^{-n}$ and 
Riemann zeta functions $\zeta(n)=\zeta(n,1)$, we obtain 
\bea
g_{-n}(S_2)=\left(1-\frac1{d_2^n}\right)\zeta(n)-\frac1{d_2^n}\sum_{k_1=1}^
{d_2-1}\zeta(n,k_1\delta),\quad n\ge 2.\label{a20}
\eea
Interchange the generators $d_1$ and $d_2$ in (\ref{a20}) and get an alternative
expression for $g_{-n}(S_2)$:  
\bea
g_{-n}(S_2)=\left(1-\frac1{d_1^n}\right)\zeta(n)-\frac1{d_1^n}\sum_{k_2=1}^
{d_1-1}\zeta\left(n,\frac{k_2}{\delta}\right).\label{a21}
\eea
\section{Identities for Hurwitz zeta functions}\label{w4}
Combining formulas (\ref{a20}) and (\ref{a21}), we get the identity 
\bea
\frac1{d_2^n}\left[\zeta(n)+\sum_{k=1}^{d_2-1}\zeta\left(n,k\frac{d_1}{d_2}
\right)\right]=\frac1{d_1^n}\left[\zeta(n)+\sum_{k=1}^{d_1-1}\zeta\left(n,k
\frac{d_2}{d_1}\right)\right],\quad n\ge 2,\label{a22}
\eea
that yields an expression for another important function for numerical semigroup
$S_2$, namely, a sum of the negative degrees of non-gaps values (zero excluding)
 $G_{-n}(S_2)=\sum_{s\in S_2\setminus \{0\}}s^{-n},n\ge 2$. This conclusion 
follows from a simple identity $g_{-n}(S_2)+G_{-n}(S_2)=\zeta(n)$ and formulas 
(\ref{a20},\ref{a21}).

Another spinoff of formulas (\ref{a20}) and (\ref{a21}) is a set of identities 
for the Hurwitz zeta functions. Consider the numerical semigroup $S_2=\la 2,3
\ra$ with one gap ${\mathbb N}\setminus\la 2,3\ra=\left\{1\right\}$. 
Substituting $g_{-n}=1$ into formulas (\ref{a20}) and (\ref{a21}), we obtain 
\bea
\zeta\left(n,\frac{2}{3}\right)+\zeta\left(n,\frac{4}{3}\right)=(3^n-1)\zeta(n)
-3^n\quad\hbox{and}\quad\zeta\left(n,\frac{3}{2}\right)=(2^n-1)\zeta(n)-2^n,
\nonumber
\eea
respectively. Making use of semigroups with more complex set of gaps, we can 
get more sophisticated identities. For example, consider the numerical semigroup
$S_2=\la 3,4\ra$ with three gaps ${\mathbb N}\setminus\la 3,4\ra=\left\{1,2,5
\right\}$. Substituting it into (\ref{a20}) and (\ref{a21}), we have 
\bea
\zeta\left(n,\frac{3}{4}\right)+\zeta\left(n,\frac{6}{4}\right)+\zeta\left(n,
\frac{9}{4}\right)=(4^n-1)\zeta(n)-\left(4^n+2^n+\left(\frac{4}{5}\right)^n
\right)\nonumber
\eea 
and 
\bea 
\zeta\left(n,\frac{4}{3}\right)+\zeta\left(n,\frac{8}{3}\right)=(3^n-1)\zeta(n)-
\left(3^n+\left(\frac{3}{2}\right)^n+\left(\frac{3}{5}\right)^n\right),\nonumber
\eea
respectively.  
\section*{Acknowlegement}
The main part of the paper was written during the stay of one of the authors 
(LGF) at the School of Mathematics and Statistics of Wuhan University and its 
hospitality is highly appreciated. The research was supported in part (LGF) by 
the Kamea Fellowship.

\end{document}